\documentclass[10pt]{amsart}

\usepackage{amssymb, dsfont} %Math symbols
\usepackage{fullpage, graphicx} %Formatting
\usepackage{hyperref} %Clickable links
\usepackage{enumitem} %Enumerate options
\usepackage{xcolor} %For the colours
\usepackage{thmtools} %Changing the QED symbol, etc.
\usepackage{caption, subcaption} %Captioning of tables and such

\hypersetup{colorlinks=true, urlcolor=purple, citecolor=blue, linkcolor=red, pdfstartview={XYZ null null 0.90}}

\DeclareMathOperator{\nrd}{nrd} %Reduced norm
\DeclareMathOperator{\disc}{disc}
\DeclareMathOperator{\discrd}{discrd}
\DeclareMathOperator{\Nm}{Nm}

\DeclareMathOperator{\Mat}{Mat}

\DeclareMathOperator{\PSL}{PSL}

\DeclareMathOperator{\PSU}{PSU}
\DeclareMathOperator{\SU}{SU}
\DeclareMathOperator{\Tr}{Tr}

\DeclareMathOperator{\Ot}{O}

\DeclareMathOperator{\inter}{int}
\DeclareMathOperator{\exter}{ext}
\DeclareMathOperator{\red}{red}
\DeclareMathOperator{\rad}{rad}

\newcommand{\Z}{\ensuremath{\mathbb{Z}}}
\newcommand{\R}{\ensuremath{\mathbb{R}}}
\newcommand{\Q}{\ensuremath{\mathbb{Q}}}

\newcommand{\uhp}{\ensuremath{\mathbb{H}}} %Upper half plane
\newcommand{\Guhp}{\ensuremath{\Gamma\backslash\uhp}} %Gamma quotient H
\newcommand{\ud}{\ensuremath{\mathbb{D}}} %Unit disc
\newcommand{\gamphi}{\ensuremath{\Gamma^{\phi}}}
\newcommand{\Ord}{\ensuremath{\mathrm{O}}}

\newcommand{\sm}[4]{\ensuremath{\left(\begin{smallmatrix} #1 & #2\\#3 & #4\end{smallmatrix}\right)}}%Small 2x2 matrix
\newcommand{\genmtx}{\sm{a}{b}{c}{d}}%Generic matrix [a,b;c,d] as a small matrix
\newcommand{\spa}[1]{\ensuremath{\left\langle #1\right\rangle}}

\newtheorem{theorem}{Theorem}[section]

\newtheorem{proposition}[theorem]{Proposition}

\theoremstyle{definition}
\newtheorem{algorithm}[theorem]{Algorithm}
\newtheorem{definition}[theorem]{Definition}

\newtheorem{remark}[theorem]{Remark}
\newtheorem{heuristic}[theorem]{Heuristic}
\newtheorem{observation}[theorem]{Observation}

\numberwithin{equation}{section}

\begin{document}

%Document Info
\title[Fundamental domains]{Improved computation of fundamental domains for arithmetic Fuchsian groups}
\author[J. Rickards]{James Rickards}
\address{University of Colorado Boulder, Boulder, Colorado, USA}
\email{james.rickards@colorado.edu}
\urladdr{https://math.colorado.edu/~jari2770/}
\date{\today}
\thanks{I thank John Voight and Aurel Page for their useful discussions and comments. I also thank Bill Allombert for his help and suggestions with PARI/GP, and the anonymous reviewers for their helpful comments. This research was supported by an NSERC Vanier Scholarship at McGill University. The author is currently partially supported by NSF-CAREER CNS-1652238 (PI Katherine E. Stange).}
\begin{abstract}
A practical algorithm to compute the fundamental domain of an arithmetic Fuchsian group was given by Voight, and implemented in Magma. It was later expanded by Page to the case of arithmetic Kleinian groups. We combine and improve on parts of both algorithms to produce a more efficient algorithm for arithmetic Fuchsian groups. This algorithm is implemented in PARI/GP, and we demonstrate the improvements by comparing running times versus the live Magma implementation.
\end{abstract}
\maketitle

\setcounter{tocdepth}{1}
\tableofcontents

\section{Introduction}
Let $\Gamma$ be a discrete subgroup of $\PSL(2, \R)$, which acts on the hyperbolic upper half plane $\uhp$. Assume that the quotient space $\Guhp$ has finite hyperbolic area $\mu(\Gamma)$, and denote the hyperbolic distance function on $\uhp$ by $d$. Let $p\in\uhp$ have trivial stabilizer under the action of $\Gamma$. Then the space
\[D(p):=\{z\in\uhp:d(z, p)\leq d(gz, p)\text{ for all $g\in\Gamma$}\}\]
forms a fundamental domain for $\Guhp$, and is known as a Dirichlet domain. It is a connected region whose boundary is a closed hyperbolic polygon with finitely many sides. Two examples of Dirichlet domains are given in Figure \ref{fig:twoexamples}; the boundaries of the domains are in green, interiors in grey, and they are displayed in the unit disc model of hyperbolic space.

\begin{figure}[ht]
	\centering
	\begin{subfigure}{.5\textwidth}
		\centering
		\includegraphics[width=.9\linewidth]{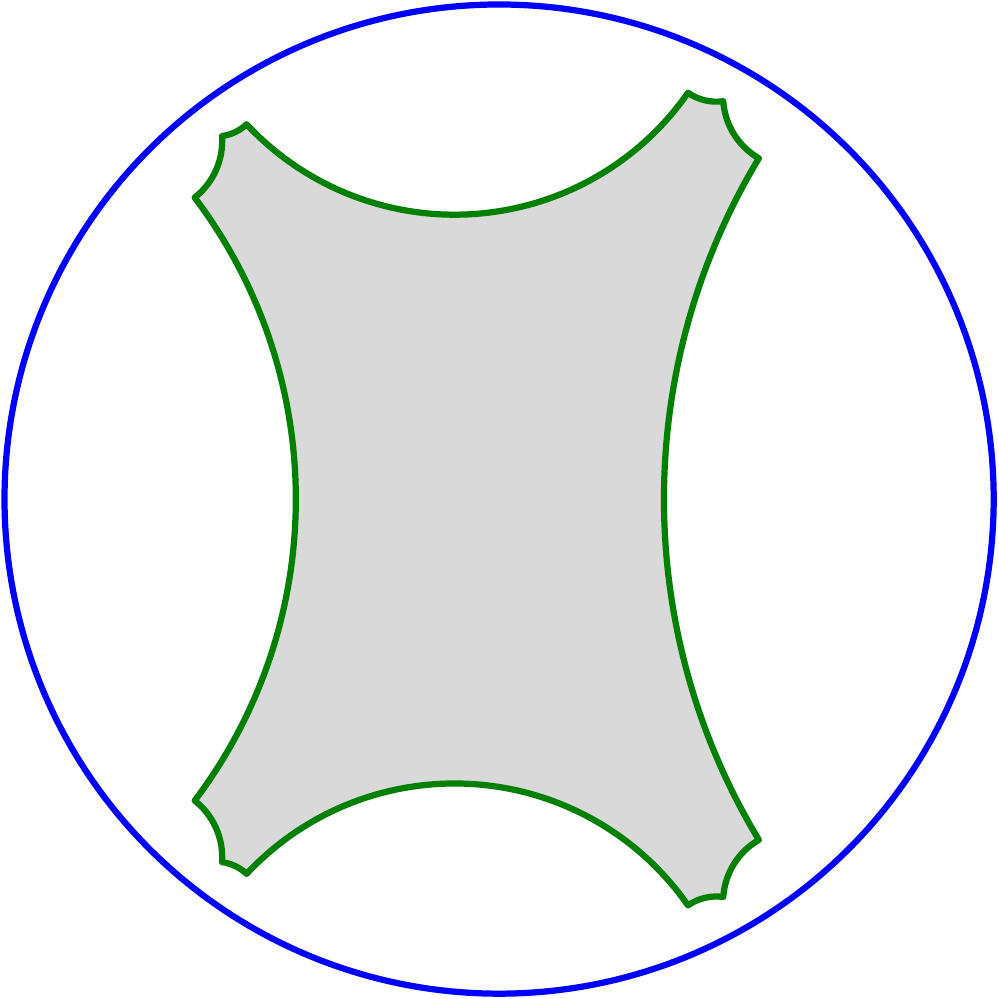}
		\caption{$F=\Q$, $\mathfrak{D}=21$.}
	\end{subfigure}%
	\begin{subfigure}{.5\textwidth}
		\centering
		\includegraphics[width=.9\linewidth]{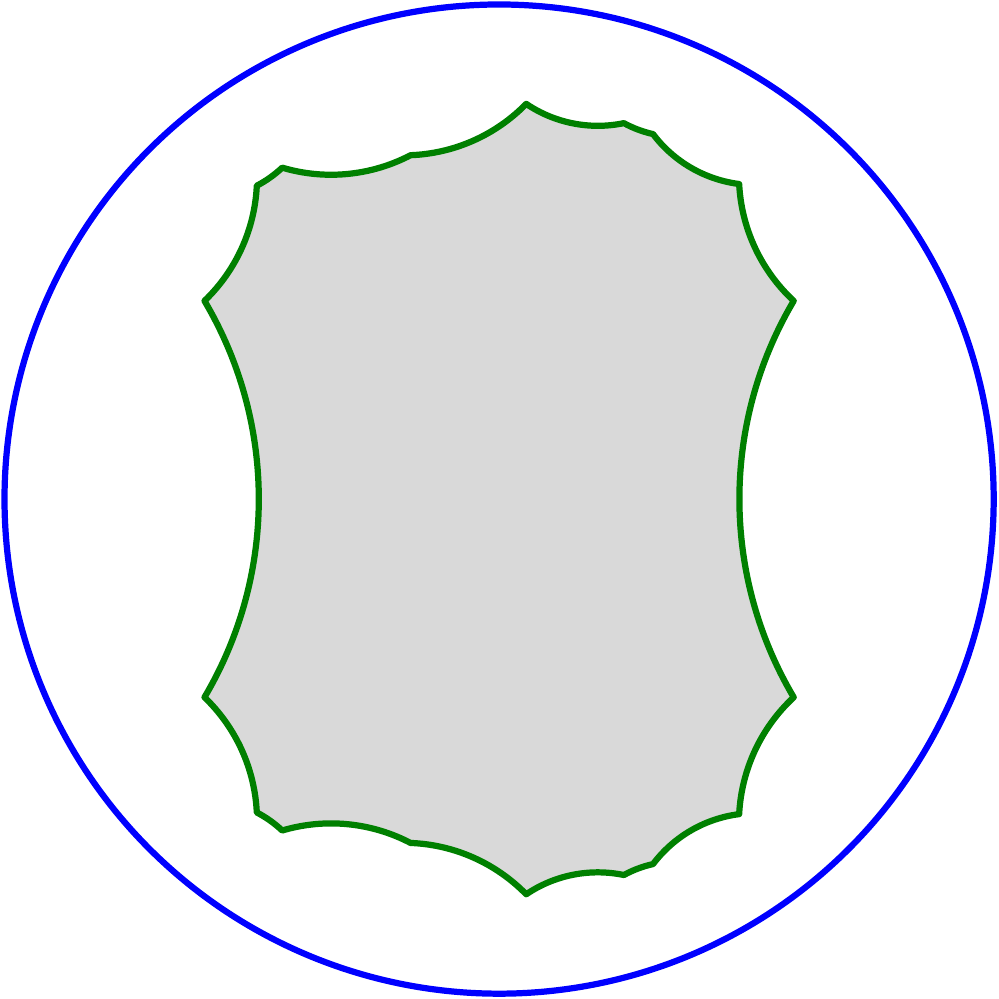}
		\caption{$F=\Q(\sqrt{5})$, $\Nm_{F/\Q}(\mathfrak{D})=61$.}
	\end{subfigure}
	\caption{Dirichlet domains corresponding to norm 1 unit groups of maximal orders in a quaternion algebra of discriminant $\mathfrak{D}$ over a number field $F$.}\label{fig:twoexamples}
\end{figure}

Explicitly computing fundamental domains has many applications, including:
\begin{itemize}
\item Computing a presentation for $\Gamma$ with a minimal set of generators (Theorem 5.1 of \cite{JV09});
\item Solving the word problem with respect to this set of generators (Algorithm 4.3 of \cite{JV09});
\item Computing Hilbert modular forms (\cite{DemV13});
\item Efficiently computing the intersection number of pairs of closed geodesics (\cite{JR21shim}).
\end{itemize}

In \cite{JV09}, John Voight published an algorithm to compute $D(p)$. The algorithm has two main parts:
\begin{itemize}
\item Element enumeration: algebraic algorithms to produce non-trivial elements of $\Gamma$, which are added to a set $G$. This step is given for \textit{arithmetic Fuchsian groups}, which are commensurable with unit groups of maximal orders in a quaternion algebra over a totally real number field that is split at exactly one real place.
\item Geometry: geometric algorithms used to compute the fundamental domain of $\spa{G}$. This step is valid for all Fuchsian groups $\Gamma$.
\end{itemize}
The process terminates once the hyperbolic area of $\spa{G}\backslash\uhp$ equals $\mu(\Gamma)$ (computed with an explicit formula).

These algorithms were implemented in Magma \cite{MAGMA}, and do a reasonable job with small cases, but do not scale very well. An improvement to the enumeration algorithms was given by Page in \cite{AP10}: he replaces the deterministic element generation by a probabilistic algorithm, which in practice performs significantly better. He also generalized the setting from arithmetic Fuchsian groups to arithmetic Kleinian groups, and has a Magma implementation available from his website.

In this paper, we aim to further improve the computation by improving the geometric part (as described by Voight), and refining the element enumeration (as described by Page). These algorithms have been implemented by the author in PARI/GP (\cite{PARI}), and are publicly available in a GitHub repository (\cite{GHfdom}). Sample running time comparing the live Magma implementation and the PARI implementation are found in Table \ref{table:rtimes1} (all computations were run on the same McGill University server). The degree and discriminant of the base number field $F$, norm to $\Q$ of the discriminant of the algebra, area, number of sides, and running times are recorded. 

\begin{table}[hbt]
\centering
\caption{Running times of the PARI versus the Magma implementation.}\label{table:rtimes1}
\begin{tabular}{|c|c|c|c|c|c|c|c|} 
\hline
$\deg(F)$ & $\disc(F)$ & $N(\mathfrak{D})$ & Area     & Sides & t(MAGMA) (s) & t(PARI) (s) & $\frac{\text{t(MAGMA)}}{\text{t(PARI)}}$\\ \hline
1         &          1 &                33 & 20.943   & 17    & 13.190       & 0.022    & 599.5   \\ \hline
1         &          1 &               793 & 753.982  & 640   & 15727.170    & 1.718    & 9154.3  \\ \hline
2         &         33 &                37 & 226.195  & 222   & 297.750      & 0.946    & 314.7   \\ \hline
2         &         44 &                79 & 571.770  & 552   & 4182.640     & 3.142    & 1331.2  \\ \hline
3         &        473 &                99 & 418.879  & 406   & 104146.830   & 4.382    & 23767.0 \\ \hline
4         &      14656 &                17 & 469.145  & 454   & 2487.800     & 12.107   & 205.5   \\ \hline
5         &    5763833 &                 1 & 4490.383 & 4294  & 2746313.540  & 1242.494 & 2210.3  \\ \hline
7         &   20134393 &               119 & 1507.964 & 1446  & 2236865.680  & 1234.850 & 1811.4  \\ \hline
\end{tabular}
\end{table}

In Section \ref{sec:geometry}, we detail the geometric portion of the algorithm, and compare the theoretical running time with Voight's paper. Taking element generation as an oracle, we give the general algorithm to compute the fundamental domain in Section \ref{sec:fullalg}. Section \ref{sec:enumeration} specializes Page's enumeration to our setting, and investigates optimal selection of required constants. Finally, Section \ref{sec:time} gives more sample running times over a range of $\mu(\Gamma)$ for $\deg(F)\leq 4$.

\section{Geometry}\label{sec:geometry}

Instead of working in the upper half plane, it is better to work with the unit disc model, $\ud$ (final results can be transferred back if desired). To this end, fix a $p\in\uhp$ which has trivial stabilizer under the action of $\Gamma$. Consider the map $\phi:\uhp\rightarrow\ud$ given by
\[\phi(z)=\dfrac{z-p}{z-\overline{p}},\]
which is the conformal equivalence between $\uhp$ and $\ud$ that sends $p$ to $0$. The group $\Gamma$ now acts on $\ud$ via $\Gamma^{\phi}:=\phi\Gamma\phi^{-1}\subseteq\PSU(1, 1)$.

\subsection{Normalized boundary}

\begin{definition}
For a non-identity element $g=\genmtx\in\gamphi$, define the \textit{isometric circle} of $g$ to be
\[I(g):=\{z\in\ud:|cz+d|=1\}.\]
\end{definition}
As the point $p$ has trivial stabilizer under $\Gamma$, $c\neq 0$, hence $I(g)$ is an arc of the circle with radius $\frac{1}{|c|}$ and centre $\frac{-d}{c}$. By convention, the arc runs counterclockwise from the initial point to the terminal point. Furthermore, define
\[\inter(I(g)):=\{z\in\ud:|cz+d|<1\},\qquad\exter(I(g)):=\{z\in\ud:|cz+d|>1\},\]
to be the interior and exterior of $I(g)$ respectively. See Figure \ref{fig:Ig} for an example.

\begin{figure}[htb]
	\includegraphics{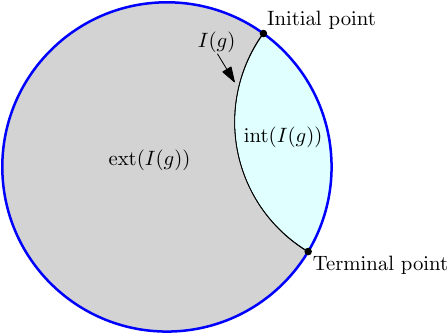}
	\caption{$I(g)$ example}\label{fig:Ig}
\end{figure}

The isometric circle has the following alternate characterization:
\[d(z,0)<d(gz, 0)\qquad\text{ if and only if }\qquad z\in\exter(I(g)).\]
The analogous statements with $<$ replaced by $=$ or $>$ and $\exter(I(g))$ replaced by $I(g)$ or $\inter(I(g))$ respectively also hold (Proposition 1.3b of \cite{JV09}). 

\begin{definition}
Let $G\subset\gamphi\backslash\{1\}$ be a subset, and define the \textit{exterior domain} of $G$ to be
\[\exter(G)=\overline{\bigcap_{g\in G}\exter(I(g))}.\]
In particular, $D(0)=\exter(\gamphi\backslash\{1\})$. 
\end{definition}

Following Voight, let $G\subset\gamphi\backslash\{1\}$ be finite, and let $E=\exter(G)$. Then $E$ is bounded by a generalized hyperbolic polygon, with sides being:
\begin{itemize}
\item subarcs of $I(g)$ for $g\in G$ (proper side);
\item arcs of the unit circle (infinite side);
\end{itemize}
and vertices being:
\begin{itemize}
\item intersections of $I(g)$ with $I(g')$ for $g\neq g'$ (proper vertex);
\item intersections of $I(g)$ with the unit circle (vertex at infinity). 
\end{itemize}
The polygon $E$ can be neatly expressed via its normalized boundary.

\begin{definition}
A \textit{normalized boundary} of $E$ is a sequence $U=g_1, g_2, \ldots, g_k$ such that:
\begin{itemize}
\item $E=\exter(U)$;
\item The counterclockwise consecutive proper sides of $E$ lie on $I(g_1), I(g_2), \ldots, I(g_k)$;
\item The vertex $v$ with minimal argument in $[0, 2\pi)$ is either a proper vertex with $v\in I(g_1)\cap I(g_2)$, or a vertex at infinity with $v\in I(g_1)$.
\end{itemize}
\end{definition}

\begin{figure}[ht]
\centering
\begin{minipage}{.5\textwidth}
  \centering
  \includegraphics[width=.9\linewidth]{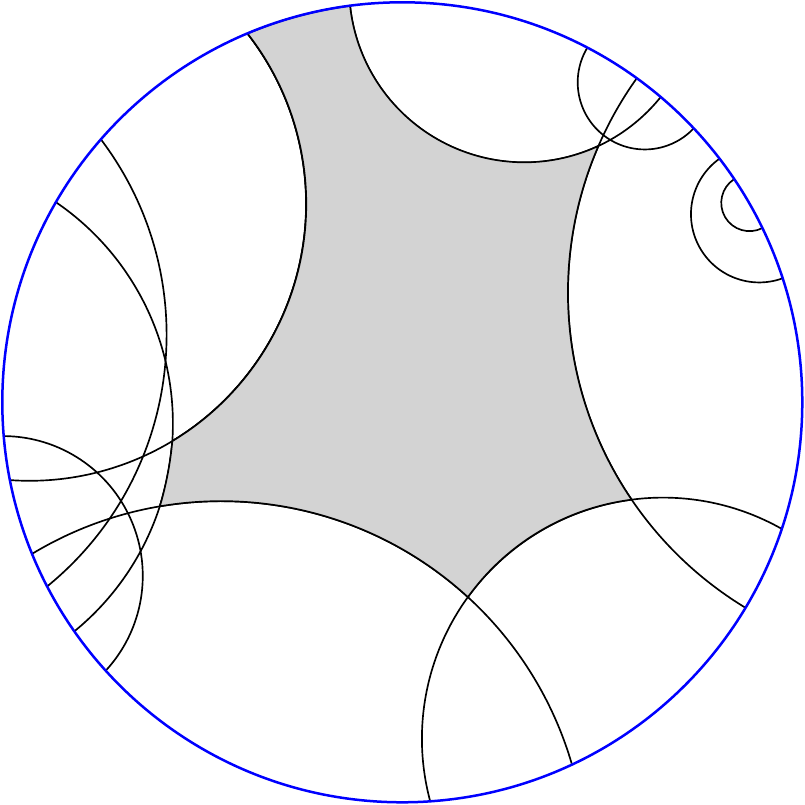}
  \captionof{figure}{$I(g)$ for all $g\in G$}\label{fig:normbound1}
\end{minipage}%
\begin{minipage}{.5\textwidth}
  \centering
  \includegraphics[width=.9\linewidth]{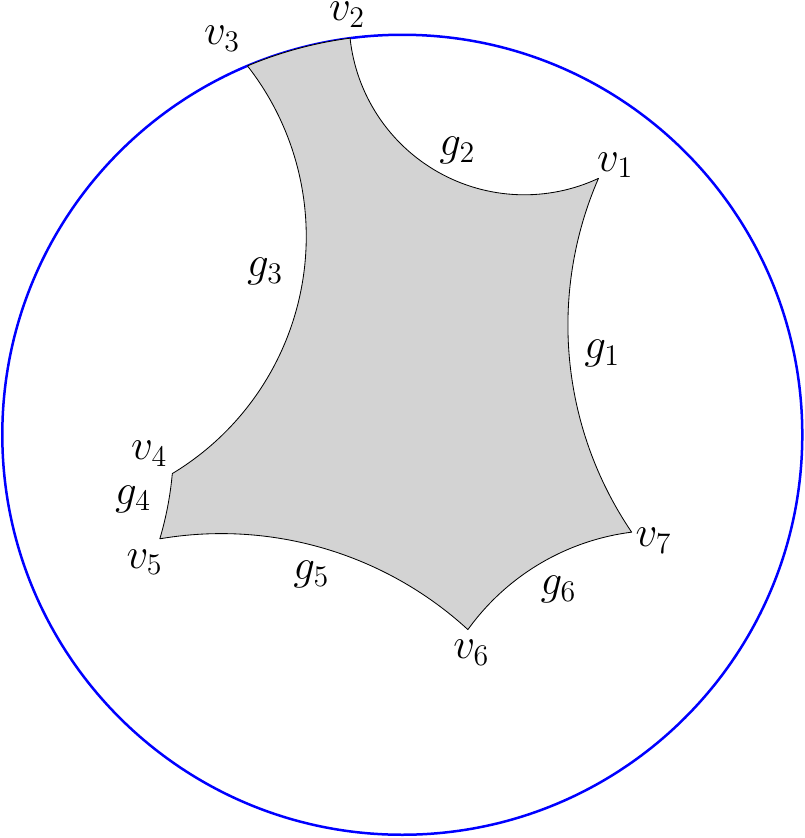}
  \captionof{figure}{Normalized boundary of $G$}\label{fig:normbound2}
\end{minipage}
\end{figure}

See Figures \ref{fig:normbound1} and \ref{fig:normbound2} for an example of going from the set of $I(g)$ for $g\in G$ to its normalized boundary. While this process is ``obvious'' visually, we require an algorithm that a computer can execute. Algorithm 2.5 of \cite{JV09} accomplishes this task, and it can be summarized as follows: start by intersecting all $I(g)$ with $[0,1]$ to find $g_1$. Next, given $g_1,g_2,\ldots,g_i$, find all intersections of $I(g_i)$ with $I(g)$ for $g\in G$, and choose the ``best one''. Repeat until $g_i=g_1$.

Taking $|G|=n$, the running time is $\Ot(nk)$, where $k$ is the number of sides of $U$. In practice, $k=\Ot(n)$ is typical, which gives a running time of $\Ot(n^2)$.

To improve this algorithm, consider the terminal points of $I(g_i)$, and observe that they are in order around the boundary of $\ud$! This follows from the fact that $I(g)$ and $I(g')$ cannot intersect twice in $\ud$. In particular, we can start by sorting $G$ by the arguments of the terminal points of $I(g)$, to get $g_1, g_2, \ldots, g_n$. The final normalized boundary will then be $g_{i_1}, g_{i_2}, \ldots, g_{i_k}$, where $i_1<i_2<\cdots<i_k$. We can iteratively construct this sequence $i_1, i_2,\ldots, i_k$ in $\Ot(n)$ steps, for a total running time of $\Ot(n\log(n))$ (due to the sorting).

\begin{algorithm}\label{alg:normbound}
Given a finite subset $G\subset\gamphi\backslash\{1\}$, this algorithm returns the normalized boundary of $\exter(G)$ in $\Ot(n\log(n))$ steps. The letters $U$ and $V$ track sequences of elements of $G$ and points in $\ud$ respectively.
\begin{enumerate}
\item Sort $G$ by the arguments (in $[0,2\pi)$) of the terminal points, to get $G=g_1, g_2, \ldots, g_n$ (take indices modulo $n$).
\item Let $H$ be the set of all $g$ such that $I(g)$ intersects $[0, 1]$. If $H$ is empty, go to step 3. Otherwise, go to step 4.
\item Let $U=g_1$, let $i=2$, let $V$ be the terminal point of $I(g_1)$, and continue to step 5.
\item Let $H'\subseteq H$ be the (non-empty) subset which gives the smallest intersection with $[0, 1]$. Let $g\in H'$ give the smallest angle of intersection of $I(g)$ with $[0, 1]$. Cyclically shift $G$ so that $g_1=g$, take $U=g_1$, let $i=2$, and let $V$ be the intersection of $I(g_1)$ with $[0, 1]$.
\item If $i=n+2$, delete $g_1$ from the end of $U$, and return $U$.
\item Assume $U$ ends with $g$, $V$ ends with $v$, and compute $v'=I(g)\cap I(g_i)$.
\begin{enumerate}
\item If $v'$ does not exist, then compare the terminal point of $I(g_i)$ with $I(g)$. If it lies in the interior, then increment $i$ by $1$, and go back to step 5. Otherwise, append $g_i$ to $U$, append the initial point of $I(g)$ to $V$, append the terminal point of $I(g_i)$ to $V$, increment $i$ by $1$, and go back to step 5.
\item If $v'$ is closer to the initial point of $I(g)$ than $v$, then append $g_i$ to $U$, append $v'$ to $V$, increment $i$ by $1$, and go back to step 5.
\item Otherwise, delete $g$ from the end of $U$, delete $v$ from the end of $V$, and go back to step 5.
\end{enumerate}
\end{enumerate}
\end{algorithm}
\begin{proof}
If $I(g)$ does not intersect $[0, 1]$ for all $g\in G$, then $g=g_1$ must minimize the argument of the terminal point of $I(g)$. Otherwise, $g=g_1$ minimizes the intersection with $[0, 1]$, and taking the smallest angle (if this is not unique) guarantees the selection. This is the content of steps 1-4.

For the rest of the algorithm, $U$ is the normalized boundary of $g_1, g_2, \ldots, g_{i-1}$, and $V$ tracks the intersection of $g$ with the previous side of $U$, where this side may be infinite, or $[0, 1]$ if $i=2$. We intersect $I(g_i)$ with $I(g)$, and if there is no intersection, then we are either enclosed inside $I(g)$, or there is an infinite side (which must exist in all subsequent iterations due to the sorting of $G$). If $v'$ is closer to the initial point of $I(g)$ than $v$, then we simply must add this new side in. Otherwise, this implies that the new side completely encloses the previous side, so backtrack by deleting the previous side and try again.

Our choice of $g_1$ guarantees that it will be a part of any partial normalized boundary, hence the sets $U$ and $V$ will never be empty. Furthermore, once we get to $i=n+1$, we still need to go on, since $g_{n+1}=g_1$ may completely enclose some of the final isometric circles in $U$. Once we have finished with this step, we have completed the circle, and are left with the normalized boundary.

As for the running time, the initial sorting is the bottleneck. If we assume that we start with a sorted $G$, then finding $g_1$ takes $\Ot(n)$ steps, and the rest of the algorithm also takes $\Ot(n)$ steps. Indeed, if we deleted $e$ elements from $U$ in a step, then we performed $e+1$ intersections. Since each element of $G$ can be added to $U$ at most once, and each element of $U$ can be deleted at most once, the result follows.
\end{proof}

Assume that Figure \ref{fig:normbound1} contains $I(g_1), I(g_2), \ldots, I(g_{11})$, ordered by argument of terminal points, with $I(g_1)$ intersecting $[0,1]$. Figure \ref{fig:normboundalg} contains the partial normalized boundaries computed by Algorithm \ref{alg:normbound}, and which values of $i$ they correspond to.

\begin{figure}[htb]
	\includegraphics{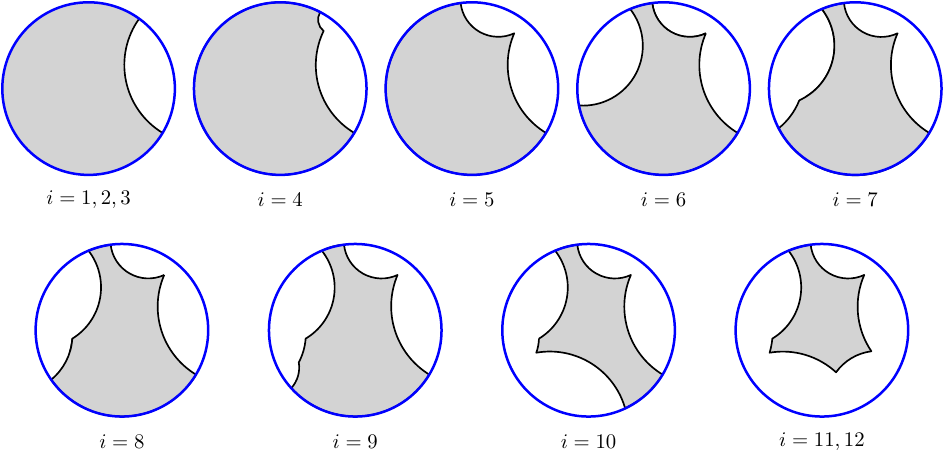}
	\caption{Partial normalized boundaries of Figure \ref{fig:normbound1} computed with Algorithm \ref{alg:normbound}}\label{fig:normboundalg}
\end{figure}

The order of magnitude of improvement is typically $\Ot\left(\frac{n}{\log(n)}\right)$, which is a substantial improvement since Algorithm \ref{alg:normbound} is called a large number of times. Furthermore, we will often be in the situation where we have a computed normalized boundary $E$ and a set $G'$, and we want to compute the normalized boundary of $E\cup G'$. Instead of blindly running Algorithm \ref{alg:normbound} on $E\cup G'$, we can save time by ``remembering'' that $E$ is a normalized boundary. We only have to sort $G'$, and combine this sorted list with the presorted $E$. When we are iteratively going through the new sequence, we can copy the data from $E$ when we are away from new sides coming from $G'$. This extra computational trick is most effective when the size of $G'$ is a much smaller than the size of $E$.

\begin{remark}
John Voight has made similar observations with regards to optimizing the geometric algorithms (personal communication). This code was implemented in Magma, but is not in the live version.
\end{remark}

\subsection{Reduction of points}

Given an enumeration of elements of $\Gamma$, Algorithm \ref{alg:normbound} is enough to compute the fundamental domain, since eventually we will have the boundary of $D(0)$ in our set $G$. However, this is far too slow in practice. One of the key ideas introduced by Voight in \cite{JV09} is the ability to efficiently compute $\exter(\spa{G})$, not just $\exter(G)$. By only requiring a set of generators, we need to enumerate less elements, which brings the total computation time to reasonable levels.

\begin{definition}
Let $G\subset\gamphi\backslash\{1\}$ be finite. Call $U$ a \textit{normalized basis} of $G$ if $U$ is the normalized boundary for $\exter(\spa{G})$.
\end{definition}

Before describing the computation of the normalized basis, we consider the reduction of elements/points to a normalized boundary. Take $G$ as above, let $z\in\inter(\ud)$, and consider the map
\begin{align*}
\rho:\Gamma\rightarrow & \R^{\geq 0}\\
\gamma \rightarrow & \rho(\gamma;z)=d(\gamma z, 0).
\end{align*}

\begin{definition}
An element $\gamma\in\Gamma$ is $(G, z)-$\textit{reduced} if for all $g\in G$, we have $\rho(\gamma;z)\leq \rho(g\gamma;z)$.
\end{definition}

Algorithm 4.3 of \cite{JV09} describes a process to reduce an element $\gamma\in\Gamma$, namely:
\begin{enumerate}
\item Compute $\min_{g\in G}\rho(g\gamma;z)$. If this is greater than or equal to $\rho(\gamma;z)$, return $\gamma$.
\item Replace $\gamma$ by $g\gamma$, where $g$ is the first element that attains the above minimum. Return to step 1.
\end{enumerate}
This algorithm terminates to produce an element $\red_G(\gamma;z):=\delta\gamma$, where $\delta\in\spa{G}$ and $\red_G(\gamma;z)$ is $(G, z)-$reduced. If $z=0$, write $\red_G(\gamma)$ for the reduction.

Proposition 4.4 of \cite{JV09} shows that if $U$ is the normalized basis of $G$, then for almost all $z\in\ud$, the element $\red_U(\gamma;z)$ is independent of all choices made. Furthermore, $\red_U(\gamma)=1$ if and only if $\gamma\in\spa{G}$. Thus, if $\gamma\in\spa{G}$, this algorithm is a way to write $\gamma^{-1}$ as a word in $U$.

While the above algorithm is valid for all finite sets $G\subset\gamphi\backslash\{1\}$, it turns out that whenever we want to reduce an element, we will have the normalized boundary of $G$, denoted $U$, at our disposal! We make the following observations:
\begin{itemize}
\item When replacing $\gamma$ with $g\gamma$, we do not need to have $\rho(g\gamma;z)=\min_{g'\in G}\rho(g'\gamma;z)$, all that is required is $\rho(g\gamma;z)<\rho(\gamma;z)$. The algorithm will still terminate in approximately the same number of steps;
\item By definition, $\rho(g\gamma;z)<\rho(\gamma;z)$ if and only if $\gamma z\in\inter(I(g))$;
\item If $\gamma z\in\inter(I(g))$ for some $g\in G$, then the straight line from $0$ to $\gamma z$ will intersect $U$ on the proper side that is part of $I(g)$ for a possible $g$. See Figure \ref{fig:reductionstep} for a demonstration of this fact.
\end{itemize}

\begin{figure}[ht]
	\includegraphics{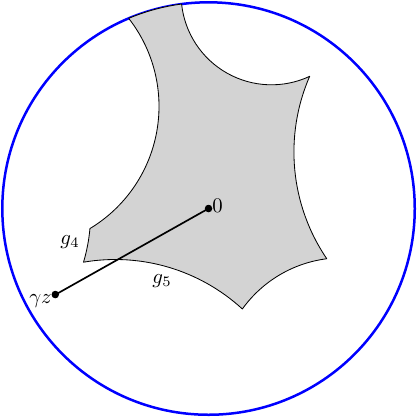}
	\caption{$\gamma z$ is in $\inter(I(g_4))$ and $\inter(I(g_5))$, and the line from $0$ to $\gamma z$ passes through $I(g_5)$.}\label{fig:reductionstep}
\end{figure}

By finding where $\arg(\gamma z)$ should be inserted in the list of arguments of vertices of $G$ (precomputed), we can determine a possible $g$ for each step (or show that we are done) in $\Ot(\log(n))$ steps. This is a large time save compared to the previous algorithm, where each step took $\Ot(n)$ steps.

\begin{algorithm}\label{alg:reduct}
Let $G$ be the normalized boundary of a finite subset of $\Gamma$, let the vertices of $G$ be $v_1, v_2, \ldots, v_k$, let $z\in\inter(\ud)$, and let $\gamma\in \Gamma$. The following steps return $\delta$, where $\red_G(\gamma;0)=\delta\gamma$ and $\delta\in\spa{G}$:
\begin{enumerate}
\item Initialize $\delta=1$, and $z'=\gamma z$.
\item Use a binary search to determine the index $i$ such that $\arg(v_i)\leq \arg(z')<\arg(v_{i+1})$.
\item Let $w$ be the intersection point of the side $v_iv_{i+1}$ (corresponding to $I(g)$) and the straight line (with respect to Euclidean geometry) between $0$ and $z'$.
\item If $|z'|\leq |w|$, return $\delta$. Otherwise, replace $(\delta, z')$ by $(g\delta, gz')$, and return to step 2.
\end{enumerate}
\end{algorithm}

\subsection{Side pairing}

One crucial omission thus far is the notion of a side pairing for Dirichlet domains. Let $P$ be a generalized hyperbolic polygon (allowing infinite sides), with proper sides $S$. Consider the set
\[\text{SP}=\{(g, s, s'):s'=g(s)\}\subseteq\Gamma\times S\times S,\]
and call two sides $(s, s')$ paired if there exists a $g\neq 1$ for which $(g, s, s')\in\text{SP}$, where we allow a side to be paired with itself (note that that this relation is symmetric, but not necessarily reflexive). Say $P$ has a \textit{side pairing} if this relation induces a partition of $S$ into singletons/pairs, i.e. for all $s\in S$ there exists a unique side $s'\in S$ with $(s, s')$ paired.

As noted in Proposition 1.1 of \cite{JV09}, the Dirichlet domain $D(p)$ has a side pairing. Furthermore, if $G$ is a normalized boundary whose exterior domain has a side pairing, then $\exter(G)$ is the fundamental domain of $\spa{G}$. In particular, to compute the normalized basis of a finite set $G$, it suffices to compute a normalized boundary $G'$ that has a side paring and $\spa{G'}=\spa{G}$.

In the case of Dirichlet domains, it is easy to see that if we have a side $s\subseteq I(g)$, then the only possible side $s$ could be paired with is $gs$. In particular, given a normalized boundary $G$ with vertices $v_1, v_2,\ldots,v_k$ and a proper side $s=v_iv_{i+1}\subseteq I(g)$, we can compute $g(v_i)$, search for its argument in the ordered list of $\arg(v_j)$, and then determine the $j$ such that $g(v_i)=v_j$ (or determine that no such $j$ exists). By computing $g(v_{i+1})$ and comparing it with $v_{j-1}$, we can determine if the side $s$ is paired or not. Since all operations besides the set search are $\Ot(1)$ time, the the side pairing test takes $\Ot(\log(n))$ time, with $|G|=n$. If we try to find all paired sides, this will take $\Ot(n\log(n))$ time.

\subsection{Normalized basis}
Algorithm 4.7 of \cite{JV09} details how to compute the normalized basis of $G$, and is not modified here. We do record it, to demonstrate that whenever we apply Algorithm \ref{alg:reduct} to reduce an element, the corresponding normalized boundary has been pre-computed (and thus we are able to apply our improved algorithms).

\begin{algorithm}\label{alg:normbasis}[Normalized basis algorithm, Algorithm 4.7 of \cite{JV09}]
Let $G\subset\gamphi\backslash\{1\}$ be finite. The normalized basis of $G$ can be computed as follows:
\begin{enumerate}
\item Let $G'=G\cup G^{-1}=\{g:g\text{ or }g^{-1}\in G\}$.
\item Use Algorithm \ref{alg:normbound} to compute the normalized boundary of $G'$, denoted $U$.
\item Let $G''=G'$. For each element $g\in G'$, compute $\red_U(g)$. If $\red_U(g)\neq 1$, extend $G''$ by $\red_U(g)^{-1}$.
\item Compute the normalized boundary of $G''$, denoted $U'$.
\item If $U'=U$ and $G''=G'$, then continue. Otherwise, replace $(G', U)$ by $(G'', U')$, and return to step 3.
\item Compute the set of paired sides of $U$. If this is a side pairing, then return $U$. Otherwise, for all pairs $(s, v)$ of an unpaired proper side $s$ with unpaired vertex $v$ (i.e. $s\subseteq I(g)$ and $gv$ is not a vertex of $U$), compute $\red_{G'}(g;v)$ and add it to $G''$ (if $v$ is a vertex at infinity, replace $v$ by a nearby point in the interior of $\ud$).
\item Let $G'=G''$, and return to step 2.
\end{enumerate}
\end{algorithm}

\section{The general algorithm (without enumeration)}\label{sec:fullalg}
With the geometry in hand, we describe the computation of the fundamental domain, since this will help guide us in considering the enumeration of elements of $\Gamma$.

\begin{algorithm}\label{alg:fullwithoracle}
Let $\Gamma$ be a discrete subgroup of $\PSL(2, \R)$, so that $\Gamma^{\phi}\subseteq\PSU(1, 1)$. Let $f(\Gamma)$ denote an oracle that returns a finite set of elements of $\Gamma$. This algorithm returns $D(0)$, which is a fundamental domain of $\Gamma$.
\begin{enumerate}
\item Compute $\mu(\Gamma)$ via theoretical means. Initialize $G=\emptyset$.
\item Call $f(\Gamma)$ to generate $G'\subseteq\Gamma^{\phi}$.
\item Use Algorithm \ref{alg:normbasis} to compute $U$, the normalized basis of $G'\cup G$.
\item Compute the hyperbolic area of $U$. If it is equal to $\mu(\Gamma)$, then return $U$. Otherwise, let $G=U$, and return to step 2.
\end{enumerate}
\end{algorithm}

Besides the oracle, this algorithm assumes two things: explicit computation of the hyperbolic area of $U$, and theoretical computation of $\mu(\Gamma)$. The first of these is easy to resolve: it is well known that the area of a hyperbolic polygon with $n$ sides and angles $\alpha_1, \alpha_2, \ldots, \alpha_n$ is
\[\mu(P)=(n-2)\pi-\sum_{i=1}^n \alpha_i.\]

Therefore, the only requirements to compute a fundamental domain are a way to generate (enough) elements of the group, and a computation of the area (or another means to determine when we are done). In the case of arithmetic Fuchsian groups, an enumeration is described in the next section, and the area computation is classical (see Theorem 39.1.8 of \cite{JV21} for a statement and proof).

\begin{remark}
The most expensive part of the area computation for arithmetic Fuchsian groups is the computation of $\zeta_F(2)$, where $F$ is the totally real number field. With larger fields, we need higher precision, and the cost to compute this zeta value accurately grows high. However, a precise value for $\zeta_F(2)$ is unnecessary! Indeed, as soon as we have a normalized basis with finite area, it will correspond to a finite index subgroup $\Gamma'$ of $\Gamma$. In particular, $\mu(\Gamma')=[\Gamma:\Gamma']\mu(\Gamma)\geq 2\mu(\Gamma)$ if they are not equal. Thus, it suffices to check that the hyperbolic area is less than $2\mu(\Gamma)$, which requires very little precision. In practice, it is incredibly unlikely to end up with a $\Gamma'\neq\Gamma$; the algorithm will nearly always stop once the area is finite.
\end{remark}

\section{Enumeration}\label{sec:enumeration}

Let $F$ be a totally real number field of degree $n$ with ring of integers $\mathcal{O}_F$, and let $B$ be a quaternion algebra over $F$ that is split at exactly one infinite place. Consider $B$ as being embedded in $\Mat(2, \R)$ via the unique split infinite place, and let $\Ord$ be an order of $B$. Then the group $\Gamma_{\Ord}:=\Ord^1/\{\pm 1\}\subseteq\PSL(2, \R)$, the group of units in $\Ord$ of reduced norm $1$, is an arithmetic Fuchsian group.

\begin{definition}
Label the Fuchsian group corresponding to $\Ord$ by the integer triple $(n, d, N)$, where $n$ is the degree of $F$, $d=\disc(F)$ is the discriminant of $F$, and $N=\Nm_{F/\Q}(\discrd(\Ord))$ is the norm to $\Q$ of the reduced discriminant of the order $\Ord$. Call this triple the \textit{data} associated to the group.
\end{definition}

Note that the data does not necessarily uniquely label a Fuchsian group, since distinct number fields may have the same discriminant, and nonisomorphic algebras/orders may give the same $N$.

As in Section \ref{sec:geometry}, fix $p=x+yi\in\uhp$, and let $\phi:\uhp\rightarrow\ud$ be the corresponding map sending $p$ to $0$. For $g=\genmtx\in\Mat(2, \R)$, define
\[f_g(p):=cp^2+(d-a)p-b.\]
If $g$ has norm $1$, let $g^{\phi}=\sm{A}{B}{C}{D}\in\SU(1, 1)$, and then a computation shows that
\[C=\dfrac{-\overline{f_g(p)}}{2iy},\text{ and the radius of $I(g)$ is }\dfrac{1}{|C|}.\]
In particular, if $M=\sm{A}{B}{C}{D}\in\text{U}(1, 1)$, define
\[f_M(p):=2iy\overline{C},\]
so that $f_g(p)=f_{g^{\phi}}(p)$.

\begin{definition}
Let $z_1, z_2\in\ud$, and fix $M_1, M_2\in\PSU(1, 1)$ so that $M_i(0)=z_i$ for $i=1,2$. For $g\in B$, define the quadratic form
\[Q_{z_1, z_2}(g):=\dfrac{1}{2y^2}\left|f_{M_2^{-1}g^{\phi}M_1}(p)\right|^2+\Tr_{F/\Q}(\nrd(g)).\]
\end{definition}

This is the analogue of $Q_{z_1, z_2}$ from Definition 28 in \cite{AP10}, and is similarly well defined (independent of $M_1, M_2$) and positive definite. In \cite{JV09}, Voight defined the absolute reduced norm $N$, which satisfies
\[N(g)=2y^2Q_{0, 0}(g).\]
The analogue of Proposition 30 of \cite{AP10} is Proposition \ref{prop:Qz1z2property}.

\begin{proposition}\label{prop:Qz1z2property}
If $g\in\Gamma_{\Ord}$, then
\[Q_{z_1,z_2}(g)=\cosh(d(g^{\phi}z_1, z_2))+n-1.\]
\end{proposition}
\begin{proof}
Let $M=M_2^{-1}g^{\phi}M_1=\sm{A}{B}{\overline{B}}{\overline{A}}$, with $|A|^2-|B|^2=1$. Since $\nrd(g)=1$, it suffices to prove that
\[\cosh(d(gz^{\phi}_1, z_2))-1=2|B|^2.\]
Since the action of $\PSU$ is isometric with respect to hyperbolic distance,
\[d(g^{\phi}z_1, z_2)=d(M_2^{-1}g^{\phi}z_1, M_2^{-1}z_2)=d(M0, 0).\]
Applying a classical formula for the hyperbolic distance, we get
\[\cosh(d(gz^{\phi}_1, z_2))-1=\dfrac{2|B/\overline{A}|^2}{1-|B/\overline{A}|^2}=\dfrac{2|B|^2}{|A|^2-|B|^2}=2|B|^2,\]
as desired.
\end{proof}

In particular, if we pick $z_1, z_2$ such that $gz_1$ is close to $z_2$ for some $g\in\Gamma_{\Ord}$, then by enumerating vectors $g'\in\Ord$ such that $Q_{z_1, z_2}(g')\leq C$ and checking which of them have reduced norm $1$, we can recover $g$. Since $\Ord$ is a positive definite rank $4n$ module over $\Z$, the small vectors of $Q_{z_1,z_2}$ can be enumerated with the Fincke-Pohst algorithm, \cite{FP85}.

In \cite{JV09}, the enumeration strategy was to solve $Q_{0,0}(g)\leq C$, and compute the normalized basis of the generated elements. If the area is larger than the target area, then increase $C$ and repeat. Since the boundary of the fundamental domain has finitely many sides and 
\[Q_{0,0}(g)=\dfrac{2}{\rad(I(g))^2}+n\]
for $g\in\Ord^1$, for large enough $C$ we obtain the fundamental domain. A downside to this strategy is the number of $g\in \Ord$ with $Q_{z_1, z_2}(g)\leq C$ grows like $C^{2n}$, the proportion of elements with $g\in\Ord^1$ shrinks, and thus the computation time blows up. The timing also has high variance, as the value of $C$ you need can vary greatly for similar examples.

In \cite{AP10}, Page introduced $Q_{z_1, z_2}$ (defined for Kleinian groups, which we specialized down to Fuchsian groups), and used this new freedom to define a probabilistic enumeration that greatly outperforms the deterministic one in practice. The outline of his enumeration is as follows:
\begin{itemize}
\item Pick a set of random points $Z$;
\item Solve $Q_{0, z}(g)\leq C$ for $z\in Z$;
\item Take the small vectors with $\nrd(g)=1$, and add them to your generating set;
\item Compute the normalized basis, and if the area is not correct, repeat these steps.
\end{itemize}

This strategy requires a couple of choices that are not immediately obvious:
\begin{itemize}
\item What value of $C$ is optimal?
\item How do we pick the random points $Z$, and how many of them are chosen in each pass?
\end{itemize}

Heuristics for both questions are given, but it is not clear if they remain optimal for the case of arithmetic Fuchsian groups. Furthermore, the heuristics do not specify the constants, which are necessary for an efficient practical implementation. We explore these questions, with plenty of data as evidence, in the following sections.

\begin{remark}
One subtle part of this strategy is checking if $\nrd(g)=1$. Since we are repeating this norm computation a large number of times, we pre-compute the Cholesky form of $\nrd(g)$, i.e. write it as a sum of (four) squares (see \cite{FP85} for more details). This speeds the norm computations up by a factor of $\approx 10$ over the PARI command ``algnorm''.
\end{remark}

\subsection{Choosing random points}\label{sec:pointchoice}

As in \cite{AP10}, we pick points uniformly from a hyperbolic disc of radius $R$. The choice of $R$ needs to be large enough that the random point is approximately uniform in the fundamental domain. On the other hand, if $R$ is too large, then precision issues may occur. We take $R=R_{\Ord}$, where
\[\mu(\text{disc of radius $R_{\Ord}$})=\mu(\Gamma_{\Ord})^{2.1}.\]
This is the same choice as the implementation of \cite{AP10}.

\begin{remark}\label{rem:almostdiameter}
There are several papers on the diameter of $\Guhp$, which lend support to $R_{\Ord}$ being large enough. Unconditional results are given by Chu-Li in \cite{CL16}, and conditional results on the almost-diameter are given by Golubev-Kamber in \cite{GK19}. Assuming the Selberg eigenvalue conjecture, their results imply that the almost-diameter of $\Guhp$ is bounded by
\[(1+o(1))\log(\mu(\Guhp)),\]
when $\Gamma$ is a congruence arithmetic Fuchsian group. An upcoming work by Steiner (\cite{RS21}) gives this result unconditionally in certain cases.
\end{remark}

\begin{remark}
Based on Remark \ref{rem:almostdiameter}, an exponent of $1+\epsilon$ in the definition of $R_{\Ord}$ would be sufficient to pick up the whole fundamental domain when $\Ord$ is Eichler. Since we do not want the random point to be biased, we increased the exponent to $2.1$ for $R_{\Ord}$. This is not a precise choice, merely a choice that worked sufficiently well in practice. For non-Eichler orders, a larger exponent \text{might} be required, or an alternate approach where the fundamental domain for a maximal order containing $\Ord$ is computed, and a coset enumeration algorithm is applied to push the domain down.
\end{remark}

\subsection{Choice of C}\label{sec:Cchoice}
Choosing the correct value of $C$ in the computation of $Q_{0, z}(g')\leq C$ is extremely important. If $C$ is too small, then it will take too many trials to get an element of reduced norm $1$, and if $C$ is too large, each individual trial will take too long.

First, we compute the probability of success of a trial. Assume that $z$ is chosen sufficiently randomly, take $g\in\Ord^1$, and let $x=d(g^{\phi}(0), z)$ be the distance between $g^{\phi}(0)$ and $z$. By Proposition \ref{prop:Qz1z2property}, the trial $Q_{0, z}(g')\leq C$ will find $g$ if and only if
\[\cosh(x)=\dfrac{e^x+e^{-x}}{2}\leq C+1-n.\]
Therefore we find $g$ if and only if the hyperbolic disc of radius $\cosh^{-1}(C+1-n)$ about $z$ contains $g^{\phi}(0)$. The expected number of such $g$'s is thus the area of this disc divided by $\mu(\Gamma_{\Ord})$. Since a hyperbolic disc of radius $R$ has area
\[4\pi\sinh(R/2)^2=\pi(e^R+e^{-R}-2)=2\pi(\cosh(R)-1),\]
we derive Heuristic \ref{heur:probsuccess}.

\begin{heuristic}\label{heur:probsuccess}
The expected number of elements of $\Ord^1$ that the trial $Q_{0, z}(g')\leq C$ outputs is
\[\mathbb{E}(\text{number of elements found})=\dfrac{2\pi(C-n)}{\mu(\Gamma_{\Ord})}.\]
\end{heuristic}

To demonstrate this heuristic, we run the trial $Q_{0, z}(g')\leq C$ across a range of $C$'s in two algebras. The points $z$ are chosen uniformly at random from a disc of radius $R_{\Ord}$. In Figures \ref{fig:d2_success} and \ref{fig:d3_success}, we display the results, including the straight line predicted by Heuristic \ref{heur:probsuccess} (we only count the non-trivial elements found). The $R^2$ values of the heuristic are $0.98840348$ and $0.96506302$ respectively. This data also lends support to our chosen radius being sufficiently large to model a random point of the fundamental domain.

\begin{figure}[htb]
  \centering
  \begin{minipage}{.5\textwidth}
    \centering
    \includegraphics[width=.9\linewidth]{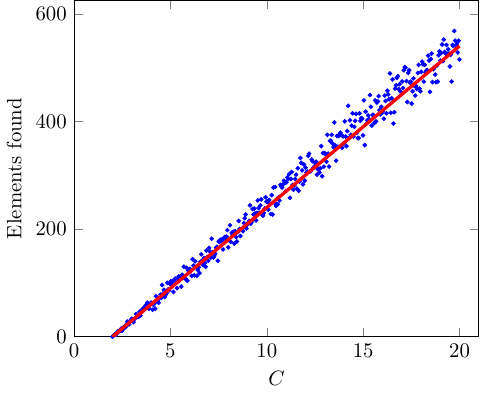}
    \captionof{figure}{\protect\raggedright Elements found in $1000$ trials of $Q_{0, z}(g')\leq C$ for curve $(2, 21, 101)$.}\label{fig:d2_success}
  \end{minipage}%NEED this to stop the line break and make it side by side
  \begin{minipage}{.5\textwidth}
    \centering
    \includegraphics[width=.9\linewidth]{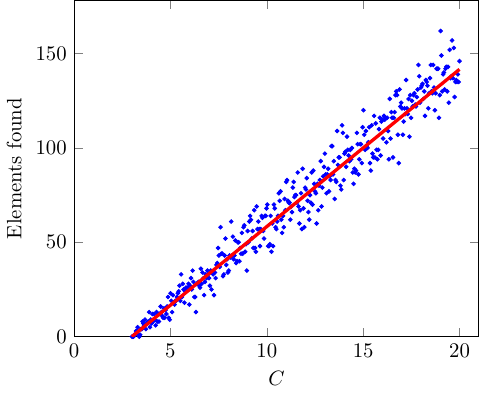}
    \captionof{figure}{\protect\raggedright Elements found in $1000$ trials of $Q_{0, z}(g')\leq C$ for curve $(3, 229, 106)$.}\label{fig:d3_success}
  \end{minipage}
\end{figure}

\begin{remark}\label{rem:oneonly}
If $C$ is small, then we will find at most one element of $\Ord$ in each trial. Indeed, if we found two elements $g_1, g_2$, then we must have $d(0, g_1^{-1}g_2(0))\leq 2\cosh^{-1}(C+1-n)$, i.e. $0$ is close to $g_1^{-1}g_2(0)$. Since the action of $\Gamma_{\Ord}^{\phi}$ is discrete, this gives a minimum bound on $C$ for this behaviour to occur. In practice, the optimal value of $C$ will typically be smaller than this, so we can stop a trial if we find a single element. Note that even if $C$ were just large enough, the second element found will be useless after the first time! An element is only useful if it is not in the span of all previously found elements. In this case, the quotient $g=g_1^{-1}g_2$ will be constant, so $g_2$ is in the span of $g_1$ and all previously found elements if we already have $g$.
\end{remark}
\begin{remark}
We can test $Q_{0,0}(g')\leq C$ at the start in an attempt to pick up some of these elements with large isometric circle radii. This is typically more useful when $\mu(\Gamma_{\Ord})$ is small, and becomes more effective as the degree of the number field grows (as the cost to generate a single element becomes high very quickly). We will use the same value of $C$ as for the tests $Q_{0, z}(g')\leq C$.
\end{remark}

To perform the enumeration of $Q_{0, z}(g')\leq C$, there are two distinct parts. Part one is the setup of Fincke-Pohst, i.e. computing the a series of matrix reductions to minimize failures in the enumeration. This part is independent of $C$. Part two consists of the actual enumeration, and we assume that the time taken is proportional to the number of vectors enumerated. This leads to Heuristic \ref{heur:timepertrial}.

\begin{heuristic}\label{heur:timepertrial}
The time to complete the enumeration of $Q_{0, z}(g')\leq C$ is
\[A+BC^{2n},\]
where $A, B$ are constants depending on $\Ord$.
\end{heuristic}

To demonstrate Heuristic \ref{heur:timepertrial}, we computed the enumeration time for $1000$ $C$'s in two quaternion algebras. Figure \ref{fig:d1_Ctime} demonstrates the results for $n=1$, along with the associated best fit curve $t=1.0560652\cdot 10^{-3}+1.1777653\cdot 10^{-8}C^2$, which gives an $R^2$ value of $0.99559127$ (the horizontal lines are due to PARI timings being integer multiples of one millisecond). Figure \ref{fig:d4_Ctime} demonstrates the results for $n=4$, along with the associated best fit curve $t=1.1927396\cdot 10^{-2}+6.3362747\cdot 10^{-14}C^8$, which gives an $R^2$ value of $0.99761548$.

\begin{figure}[htb]
\centering
\begin{minipage}{.5\textwidth}
  \centering
  \includegraphics[width=.9\linewidth]{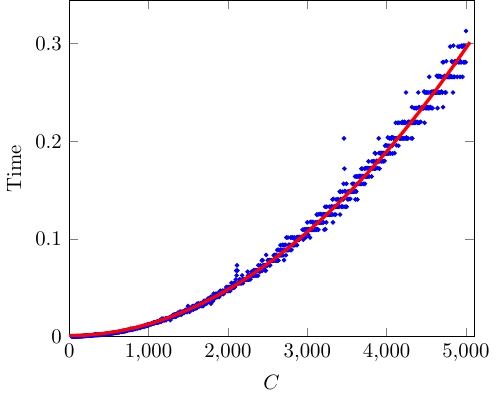}
  \captionof{figure}{\protect\raggedright Computation time of $Q_{0, z}(g')\leq C$ for curve $(1, 1, 2021)$.}\label{fig:d1_Ctime}
	\hfill
\end{minipage}%NEED this to stop the line break and make it side by side
\begin{minipage}{.5\textwidth}
  \centering
  \includegraphics[width=.9\linewidth]{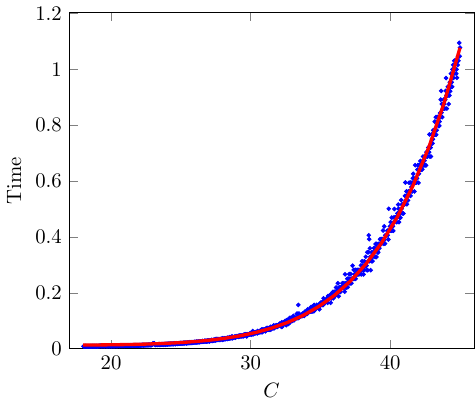}
  \captionof{figure}{\protect\raggedright Computation time of $Q_{0, z}(g')\leq C$ for curve $(4, 14013, 109)$.}\label{fig:d4_Ctime}
\end{minipage}
\end{figure}

In particular, combining Heuristics \ref{heur:probsuccess} and \ref{heur:timepertrial}, the expected time taken before a success is given by
\begin{equation}\label{eqn:expsuccess}
\left(\dfrac{\mu(\Gamma_{\Ord})B}{2\pi}\right)\dfrac{C^{2n}+A/B}{C-n}.
\end{equation}
Basic calculus implies that this is minimized for the unique real solution $C>n$ of the polynomial
\begin{equation}\label{eqn:minimalC}
(2n-1)C^{2n}-2n^2C^{2n-1}=A/B.
\end{equation}

\begin{remark}\label{rem:impdependent}
The values of $A$ and $B$ are highly dependent on factors like precision, processor speed, etc. However, these factors should affect $A$ and $B$ to approximately the same factor, leaving $A/B$ constant. On the other hand, $A/B$ is highly dependent on the implementation of Fincke-Pohst, the setup to the enumeration, or even the computation of the norm of an element.
\end{remark}

For a fixed $\Ord$, we can run similar computations to Figures \ref{fig:d1_Ctime} and \ref{fig:d4_Ctime}, and obtain approximate values for $A, B$ with a least squares regression. Solving Equation \eqref{eqn:minimalC} gives the optimal value of $C$ for this quaternion order. By performing this computation across a large selection of arithmetic Fuchsian groups, we can deduce a heuristic for the optimal value of $C$. This is presented in Heuristic \ref{heur:optC}, with the rest of this section dedicated to providing experimental evidence justifying the heuristic. It is important that the heuristic is run across a large range of $C$'s, algebra discriminants, and field discriminants, as otherwise the final constants may be significantly off.

\begin{heuristic}\label{heur:optC}
Let $\Ord$ have reduced discriminant $\mathfrak{D}$. The optimal choice of $C$ is given by
\[C_{\Ord}:=C_n\disc(F)^{1/n}\Nm_{F/\Q}(\mathfrak{D})^{1/2n},\]
for constants $C_n$. Approximate values for $C_n$ for $n\leq 8$ to 6 decimal places are:
\begin{center}
\begin{tabular}{|c|c||c|c|} 
\hline
$n$ & $C_n$        & $n$ & $C_n$        \\ \hline
1   & 2.830484 & 5   & 1.019539 \\ \hline
2   & 0.933176 & 6   & 1.018481 \\ \hline
3   & 0.909751 & 7   & 0.994256 \\ \hline
4   & 0.973456 & 8   & 0.964400 \\ \hline
\end{tabular}
\end{center}
\end{heuristic}

\begin{remark}
The non-constant part of $C_{\Ord}$ differs to the heuristic in \cite{AP10}. However, this was a typo! The Magma implementation of the algorithm uses the correct heuristic.
\end{remark}

\begin{remark}
We can also justify the heuristic theoretically (thanks to Aurel Page for the argument). Fix $n$, and consider Heuristic \ref{heur:timepertrial}. The term $BC^{2n}$ should be proportional to the number of elements enumerated, and thus $B$ is inversely proportional to the covolume of $\Ord$. Thus, $1/B$ should be proportional to $\disc(F)^2\Nm_{F/\Q}(\mathfrak{D})$. Assume $A$ is constant and $A/B$ is large, so that $C$ is also large (say $C\gg n$). Then the secondary term of Equation \eqref{eqn:minimalC} can be ignored, and Heuristic \ref{heur:optC} follows.
\end{remark}

The value of $C_n$ is dependent on implementation. If there is a future improvement in part of the implementation, then the new optimal algorithm would necessarily have larger optimal values for $C_n$. However, unless it was a very significant improvement, the current optimal values of $C_n$ would be still perform well.

The behaviour of $C_n$ is hard to explain. When $n\geq 2$, we have to deal with number theory arithmetic, which is more costly than integer arithmetic. This may explain the big drop from $C_1$ to $C_2$, but does not explain the further oscillation. For $n\geq 9$, we suggest taking $C_n=1$ as a baseline, since the data does hover around this value. When running a large computation with $n\geq 9$, it would also be worth running these experiments again for this $n$, or at least experimenting with the value of $C_n$ a bit to get a more accurate value, and thus better performance.

\subsection{Computational evidence}

To demonstrate Heuristic \ref{heur:optC}, start by fixing $F$ and varying $\mathfrak{D}$. By computing $A$ and $B$ (with a regression) as in Heuristic \ref{heur:timepertrial} and solving for $C$ with Equation \eqref{eqn:minimalC}, we can determine the optimal $C$ for each case. In Figure \ref{fig:d1_Arange}, we carry this out for 400 quaternion algebras over a quadratic field. The curve of best fit, $C=3.43356822\left(\Nm_{F/\Q}(\mathfrak{D})\right)^{1/4}$, is given in red. In Figure \ref{fig:d4_Arange}, this is done for 400 quaternion algebras in a quartic setting. The curve of best fit is $C=6.54768871\left(\Nm_{F/\Q}(\mathfrak{D})\right)^{1/8}$. 

\begin{figure}[htb]
\centering
\begin{minipage}{.5\textwidth}
  \centering
  \includegraphics[width=.9\linewidth]{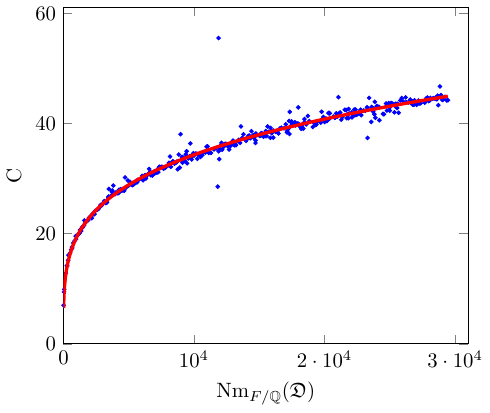}
  \captionof{figure}{\protect\raggedright $C$ for 400 quaternion algebras over $\Q(y)$, with $y^2-13=0$.}\label{fig:d1_Arange}
	\hfill
\end{minipage}%NEED this to stop the line break and make it side by side
\begin{minipage}{.5\textwidth}
  \centering
  \includegraphics[width=.9\linewidth]{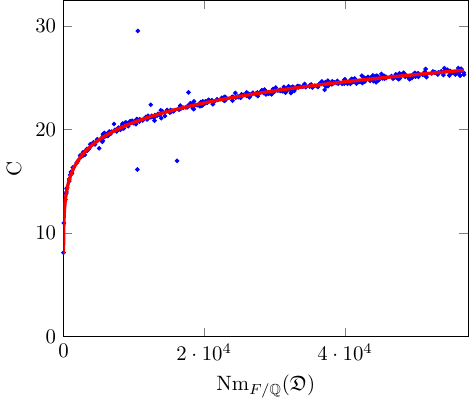}
  \captionof{figure}{\protect\raggedright $C$ for 400 quaternion algebras over $\Q(y)$, with $y^4 - 5y^2 + 5=0$.}\label{fig:d4_Arange}
\end{minipage}
\end{figure}

Next, fix $n$, and vary $F$. For each value of $n$ we take $400$ quaternion algebras over totally real number fields of degree $n$, and compute $C'=C/\Nm_{F/\Q}(\mathfrak{D})^{1/2n}$. The curves of best fit are all of the form $C'=C_n\disc(F)^{1/n}$. The data for $n=2$ and $n=3$ is displayed in Figures \ref{fig:d2_Frange} and \ref{fig:d3_Frange} respectively.

\begin{figure}[htb]
\centering
\begin{minipage}{.5\textwidth}
  \centering
  \includegraphics[width=.9\linewidth]{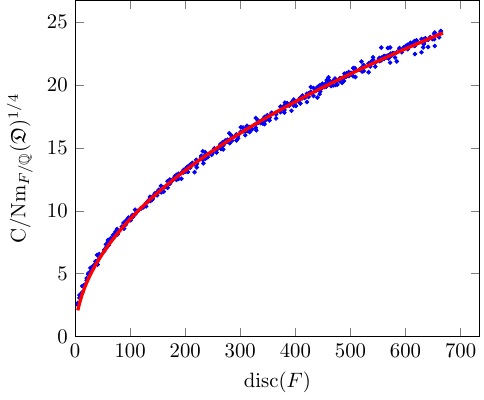}
  \captionof{figure}{\protect\raggedright $C'$ for $n=2$.}\label{fig:d2_Frange}
	\hfill
\end{minipage}%NEED this to stop the line break and make it side by side
\begin{minipage}{.5\textwidth}
  \centering
  \includegraphics[width=.9\linewidth]{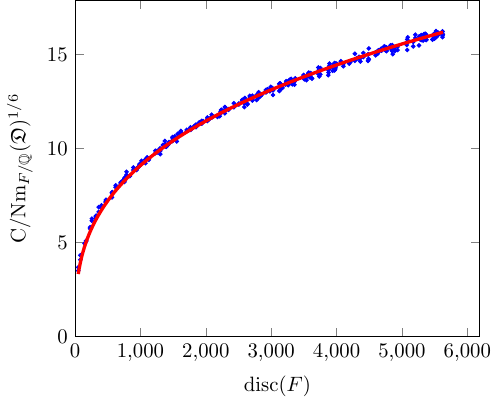}
  \captionof{figure}{\protect\raggedright $C'$ for $n=3$.}\label{fig:d3_Frange}
\end{minipage}
\end{figure}

These computations also give the values of $C_n$, as found in Heuristic \ref{heur:optC}. Since they are experimentally found, running the computations again will produce slightly different values; the values given are only intended as approximations.

To demonstrate that this theory actually works, we can fix an algebra, compute the time taken to find $N$ non-trivial elements over a range of $C$'s, and verify that Heuristic \ref{heur:optC} is close to the observed minimum. In Figures \ref{fig:d2_Nelttime} and \ref{fig:d3_Nelttime}, we take curves with data $(2, 5, 101)$ and $(3, 316, 46)$, and compute the time to generate 1000 non-trivial elements of each for 500 values of $C$ (outputting at most one value for each $z$, due to Remark \ref{rem:oneonly}). A dotted red line is drawn at the value of $C$ predicted by Heuristic \ref{heur:optC} to be minimal. In each example this value is close to the absolute minimum.

\begin{figure}[htb]
\centering
\begin{minipage}{.5\textwidth}
  \centering
  \includegraphics[width=.9\linewidth]{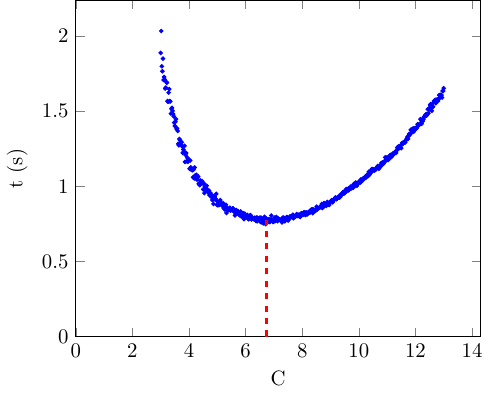}
  \captionof{figure}{Time to obtain 1000 elements for $(2, 5, 101)$.}\label{fig:d2_Nelttime}
	\hfill
\end{minipage}%NEED this to stop the line break and make it side by side
\begin{minipage}{.5\textwidth}
  \centering
  \includegraphics[width=.9\linewidth]{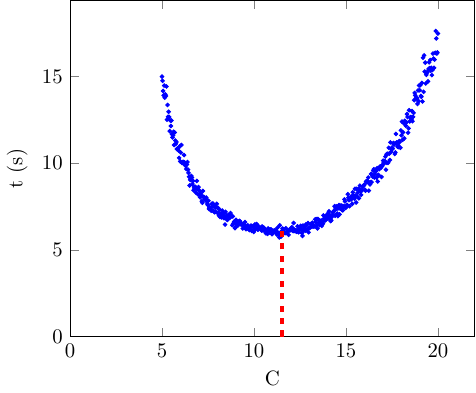}
  \captionof{figure}{Time to obtain 1000 elements for $(3, 316, 46)$.}\label{fig:d3_Nelttime}
\end{minipage}
\end{figure}

\subsection{Improved Fincke-Pohst}
Let $Q(x_1,x_2,\ldots, x_N)$ be a positive definite quadratic form. The general idea of the Fincke-Pohst enumeration of $Q(x)\leq C$ is to make a change of basis to variables $x_1', x_2',\ldots, x_N'$, write $Q$ as a sum of squares in this basis, and incrementally bound $x_{N}', x_{N-1}', \ldots, x_1'$. The change of basis was chosen in a way to minimize the number of ``dead ends'' in the incremental enumeration, and is the key part of the algorithm (for details, see Section 2 of \cite{FP85}). In our situation, we also have an (indefinite) quadratic form $Q'$ for which we require $Q'(x_1, x_2, \ldots, x_N)=1$. If we compute the change of basis for $Q'$ to the variables $x_i'$, then when we get to $x_1'$, we see that it must satisfy a quadratic equation! In particular, there are at most two possibilities for it. By solving this quadratic over the integers, we can determine if we have a solution or not. Note that we may find solutions with $Q(x)>C$, but this does not concern us since they still give an element of $\Ord^1$.

Label the classical Fincke-Pohst by ``FP'', and this new approach by ``IFP''. If we have many choices for $x_1'$, then IFP should be a faster algorithm, since we solve one quadratic equation over $\Z$ as opposed to checking the norms of a large set of elements. On the other hand, this requires a lot more (number field) arithmetic, some of which will be ``useless'' when we reach dead ends. To determine the efficacy, we compute the time required to find $1000$ non-trivial elements in a given quaternion algebra with each enumeration method (as before, stopping each trial after finding an element). We use a range of $C$'s, as the optimal value of $C$ for this approach may not be $C_{\Ord}$ (in particular, Heuristic \ref{heur:timepertrial} is no longer valid). Considering the already computed examples in Figures \ref{fig:d2_Nelttime} and \ref{fig:d3_Nelttime}, we re-compute the time taken with IFP. The output (with IFP shown in green with triangle markers) is Figures \ref{fig:d2_Nelttime_both} and \ref{fig:d3_Nelttime_both}.

\begin{figure}[hbt]
\centering
\begin{minipage}{.5\textwidth}
  \centering
  \includegraphics[width=.9\linewidth]{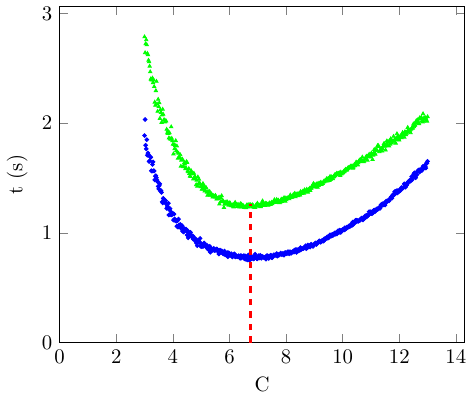}
  \captionof{figure}{Time to obtain 1000 elements for $n=2$, FP and IFP.}\label{fig:d2_Nelttime_both}
	\hfill
\end{minipage}%NEED this to stop the line break and make it side by side
\begin{minipage}{.5\textwidth}
  \centering
  \includegraphics[width=.9\linewidth]{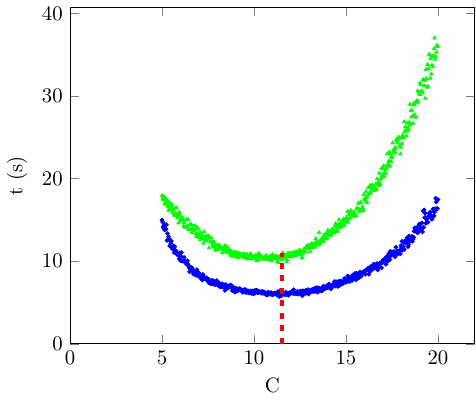}
  \captionof{figure}{Time to obtain 1000 elements for $n=3$, FP and IFP.}\label{fig:d3_Nelttime_both}
\end{minipage}
\end{figure}

In particular, IFP is slower for both examples. This fact continues to hold for all other examples attempted with $n\geq 2$. On the other hand, if $n=1$ (i.e. $F=\Q$), then IFP appears to be faster. For example, see Figures \ref{fig:d1_Nelttime_both} and \ref{fig:d1_Nelttime2_both}. The value of $C_{\Ord}$ is no longer optimal for IFP, but it is still quite reasonable and better than before. 

\begin{figure}[htb]
\centering
\begin{minipage}{.5\textwidth}
  \centering
  \includegraphics[width=.9\linewidth]{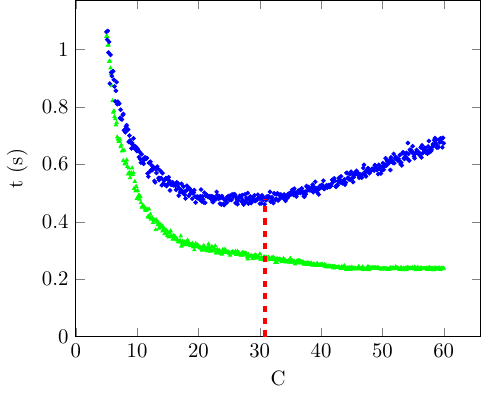}
  \captionof{figure}{Time to obtain 1000 elements for curve $(1, 1, 119)$,  FP and IFP.}\label{fig:d1_Nelttime_both}
	\hfill
\end{minipage}%NEED this to stop the line break and make it side by side
\begin{minipage}{.5\textwidth}
  \centering
  \includegraphics[width=.9\linewidth]{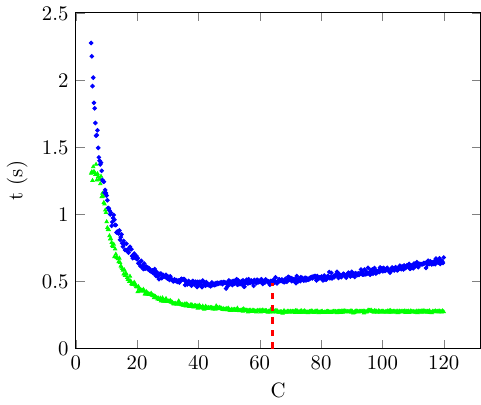}
  \captionof{figure}{Time to obtain 1000 elements for curve $(1, 1, 514)$, FP and IFP.}\label{fig:d1_Nelttime2_both}
\end{minipage}
\end{figure}

\begin{observation}
For the enumeration, we should use IFP when $n=1$, and FP when $n>1$.
\end{observation}

The difference in behaviour between $n=1$ and $n>1$ likely comes from the limiting of trials to one element. IFP thrives on cutting down large ranges for $x_1'$, and when $n>1$, we don't get these large ranges. If we instead find all norm one elements in a trial, then IFP will become more efficient than FP for large enough $C$. While this is irrelevant with regards to the current enumeration, it is useful for Voight's enumeration of elements in \cite{JV09}, which is to solve $Q_{0,0}(g)\leq C$ for increasingly large $C$.

\subsection{Balancing enumeration and geometry}\label{sec:balance}

At this point, algorithms to efficiently generate elements of $\Ord^1$ and compute normalized bases have been described. The final question is: how do we strike a balance between these two operations? If we compute the normalized basis too often, then we are wasting a lot of time with useless calls. On the other hand, if we call it less often, then we may enumerate too many elements before finishing.

One other benefit of normalized basis calls was noted in both \cite{JV09} and \cite{AP10}. If we have an infinite side of a partial fundamental domain, then an element $g\in\Ord^1$ for which $I(g)$ closes off (part) of this side can be picked up with $Q_{0,z}$ for some $z$ near the infinite side. In particular, by searching there, we can increase the chance of finding useful elements.

On the geometric side, computing the normalized basis with Algorithm \ref{alg:normbasis} should take $O(n\log(n))$ steps, where $n$ is the number of elements. To estimate the number of sides of the final fundamental domain, we do this computation for $1000$ domains with hyperbolic area at most $1000$ (distributed among $\deg(F)\leq 4$). As seen in Figure \ref{fig:sidesvarea}, the number of sides is typically proportional to the area (the analogous statement was noted for arithmetic Kleinian groups at the end of \cite{AP10}). In fact, the line of best fit is
\[\text{Number of sides}\approx 0.94747172\mu(\Gamma_{\Ord}),\]
which has an $R^2$ value of $0.99273664$. Assuming we don't compute the normalized basis \textit{too} often, this final computation should dominate. Therefore we expect the geometric part of the algorithm to grow proportionally to $\mu(\Gamma_{\Ord})\log(\mu(\Gamma_{\Ord}))$.

For the enumeration, as noted in \cite{AP10}, a number of elements proportional to $\mu(\Gamma_{\Ord})$ should generate $\Gamma_{\Ord}$ (see Theorem 1.5 of \cite{BGLS10}). Due to the probabilistic nature of the method, there is not a precise constant of proportionality for which we can guarantee success. Instead, we will get a rough idea of the proportion by computing a large number of examples. Using the same input algebras as Figure \ref{fig:sidesvarea}, we find the smallest $N$ such that the first $N$ random elements we computed were sufficient to generate the fundamental domain. This data is displayed in Figure \ref{fig:neltsvarea}, and does appear to be approximately linear, as expected. Based on this data, around $\frac{1}{4}\mu(\Gamma_{\Ord})$ elements is a good target to generate the fundamental domain.

\begin{figure}[htb]
\centering
\begin{minipage}{.5\textwidth}
  \centering
  \includegraphics[width=.9\linewidth]{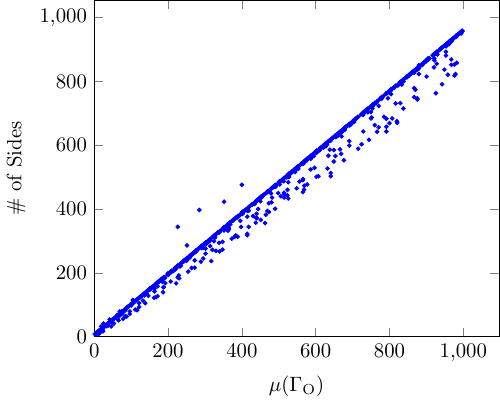}
  \captionof{figure}{Number of sides of the of the fundamental domain.}\label{fig:sidesvarea}
	\hfill
\end{minipage}%NEED this to stop the line break and make it side by side
\begin{minipage}{.5\textwidth}
  \centering
  \includegraphics[width=.9\linewidth]{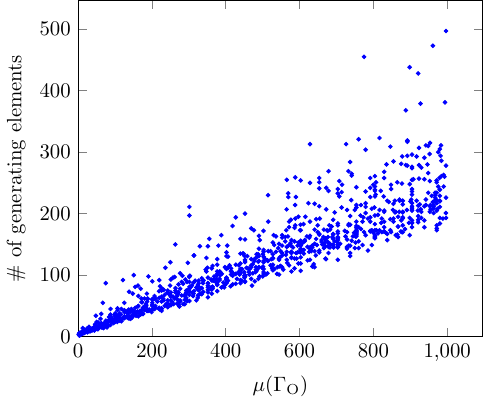}
  \captionof{figure}{Number of random elements required to generate the fundamental domain.}\label{fig:neltsvarea}
\end{minipage}
\end{figure}
%Slope=0.25645570940266703905257285147136383852, R^2=0.80982672113828457334327387213731227666 for second figure

Considering Equation \eqref{eqn:expsuccess}, we can generate an element in expected time $O(\mu(\Gamma_{\Ord}))$. Furthermore, combining the heuristics gives a computable constant $w$ such that $w\mu(\Gamma_{\Ord})^2$ random points should be close to generating the fundamental domain. 

With these heuristics in hand, we suggest choosing a constant $c<1$, and picking $cw\mu(\Gamma_{\Ord})^2$ points in each iteration. Experiments show that $c=\frac{1}{2}$ when $n=1$, and $c=\frac{1}{12}$ when $n>1$ are reasonable choices.

\begin{remark}
As seen in the previous section, the ``cost'' of choosing a poor value of $C_n$ was extremely high. On the other hand, the cost of a poor choice for the number of random centres in each iteration is much smaller. Any reasonable choice will perform decently well, and we thus do not delve as deep into the choice of $c$ in each situation as we did for the choice of $C$.
\end{remark}

\begin{remark}
Another solution to striking a balance between enumeration and geometry would be to use parallel computing. One processor would be enumerating group elements, with the other processor computing normalized bases (and supplying information on missing infinite sides). Furthermore, by having multiple processors enumerating elements, one can speed up the enumeration by a constant factor. 
\end{remark}

\section{Sample timings}\label{sec:time}

While the currently implemented algorithms correspond to the content of this paper, that may change in the future. If better approaches/constant choices/programming tricks/etc. are found, then the algorithm timings will change with it. The hardware of a computer will also affect things, and this may not be a constant factor either. In any case, these timings are a representative of the current state of affairs, and give a general guideline. All computations in this section were run on the same McGill University server as Table \ref{table:rtimes1}. This server is not particularly fast, so you will likely see similar or better times on your own machine.

\begin{figure}[hbt]
\centering
\begin{minipage}{.5\textwidth}
  \centering
  \includegraphics[width=.9\linewidth]{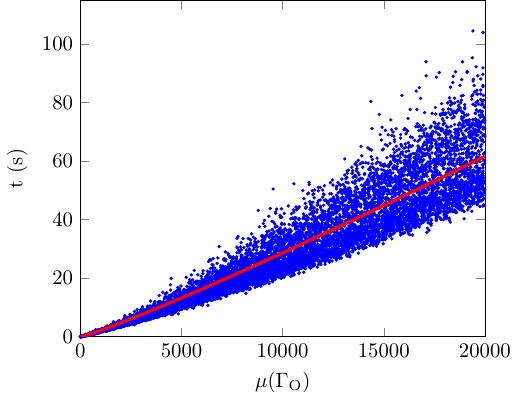}
  \captionof{figure}{Time to compute the fundamental domain, $n=1$.}\label{fig:times_d1}
	\hfill
\end{minipage}%NEED this to stop the line break and make it side by side
\begin{minipage}{.5\textwidth}
  \centering
  \includegraphics[width=.9\linewidth]{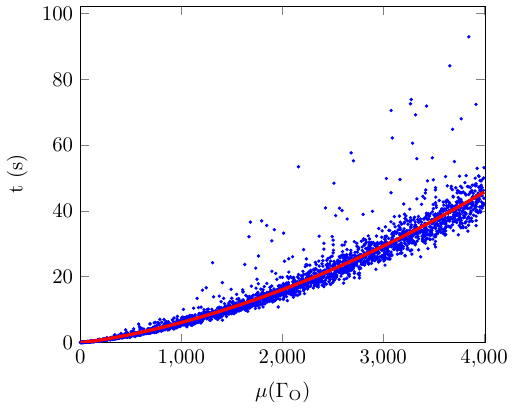}
  \captionof{figure}{Time to compute the fundamental domain, $n=2$.}\label{fig:times_d2}
\end{minipage}
\end{figure}

Considering the heuristics given in Section \ref{sec:balance}, the expected running time is
\[c_1\mu\log(\mu)+c_2\mu^2,\]
where $\mu=\mu(\Gamma_{\Ord})$, and $c_1,c_2$ depend on $n$. While the enumeration time will eventually be the most costly part of the algorithm, the normalized basis will dominate for smaller areas.

For $n=1$, we computed all fundamental domains with area at most $20000$ that correspond to maximal orders (there are $9550$ such examples). In this range, the normalized basis dominates, and the curve with $(c_1, c_2)=(0.88824714, 0)$ is shown in red in Figure \ref{fig:times_d1}.

For $n=2$, we computed $2975$ examples, all with area at most $4000$, over three fields. The curve with $(c_1, c_2)=(0.00066605054, 0.0000014754184)$ is shown in red in Figure \ref{fig:times_d2}. The enumeration becomes the dominant part of the computation time at $\mu(\Gamma_{\Ord})\approx 3700$.

\begin{figure}[hbt]
\centering
\begin{minipage}{.5\textwidth}
  \centering
  \includegraphics[width=.9\linewidth]{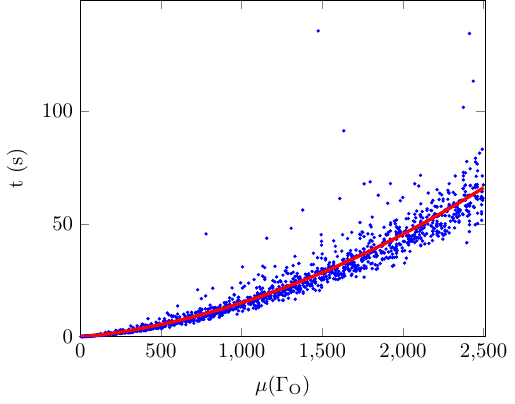}
  \captionof{figure}{Time to compute the fundamental domain, $n=3$.}\label{fig:times_d3}
	\hfill
\end{minipage}%NEED this to stop the line break and make it side by side
\begin{minipage}{.5\textwidth}
  \centering
  \includegraphics[width=.9\linewidth]{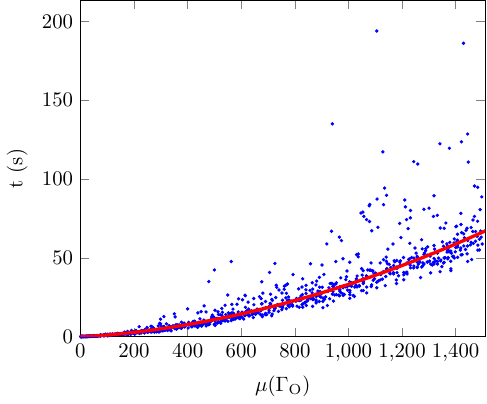}
  \captionof{figure}{Time to compute the fundamental domain, $n=4$.}\label{fig:times_d4}
\end{minipage}
\end{figure}

For $n=3$, we computed $1481$ examples, all with area at most $2500$, over four fields. The curve with $(c_1, c_2)=(0.0011822431, 0.0000068505131)$ is shown in red in Figure \ref{fig:times_d3}.

Finally, for $n=4$, we computed $912$ examples, all with area at most $1500$, over six fields. The curve with $(c_1, c_2)=(0.0018590790, 0.000021764718)$ is shown in red in Figure \ref{fig:times_d4}.

\bibliographystyle{alpha}
\bibliography{../references}
\end{document}